\documentclass[a4paper,12pt]{amsart}
\usepackage{amssymb}
\usepackage[all,cmtip]{xy}
\usepackage{amsmath}
\usepackage{chemarrow}
\usepackage{graphicx}
\usepackage[colorlinks, 
pdfborder={0 0 1}, 
linkcolor=cyan,
anchorcolor=magenta
citecolor=green
]{hyperref}
\usepackage{mathrsfs}
\usepackage{titletoc}
\usepackage{bm}
\usepackage{bbm}
\usepackage{bbold} 
\usepackage{cite}
\usepackage{subfiles}
\usepackage{extarrows} 
\usepackage[hmarginratio=1:1]{geometry}
\usepackage{colordvi}
\usepackage{color}
\usepackage{stmaryrd}
\usepackage{tcolorbox}
\usepackage{multirow} 
\usepackage{tikz-cd}

\allowdisplaybreaks

\DeclareMathOperator{\Aut}{Aut}

\DeclareMathOperator{\ch}{char}
\DeclareMathOperator{\cl}{Cl}
\DeclareMathOperator{\ccl}{CCl}

\DeclareMathOperator{\colim}{\varinjlim\limits_{n\in\mathbb{N}}}
\DeclareMathOperator{\cpr}{CPrin}

\DeclareMathOperator{\depth}{depth}

\DeclareMathOperator{\di}{Div}

\DeclareMathOperator{\Ext}{Ext}

\DeclareMathOperator{\gld}{gldim}
 
\DeclareMathOperator{\Gr}{Gr}

\DeclareMathOperator{\hh}{H}
\DeclareMathOperator{\Hom}{Hom}

\DeclareMathOperator{\height}{ht}
\DeclareMathOperator{\Img}{Im}

\DeclareMathOperator{\Kdim}{Kdim}

\DeclareMathOperator{\kk}{\Bbbk}

\DeclareMathOperator{\M}{\mathcal{M}}
\DeclareMathOperator{\m}{\mathfrak{m}}

\DeclareMathOperator{\ncl}{NCl}
\DeclareMathOperator{\npr}{NPrin}

\DeclareMathOperator{\Oz}{Oz}
\DeclareMathOperator{\p}{\mathfrak{p}}

\DeclareMathOperator{\pd}{pd}
\DeclareMathOperator{\Pic}{Pic}
\DeclareMathOperator{\Picent}{Picent}

\DeclareMathOperator{\set}{\mathcal{S}}

\DeclareMathOperator{\Spec}{Spec}

\newcommand{\bdot}{{\scalebox{0.5}{$\bullet$}}}
\numberwithin{equation}{section}

\theoremstyle{definition}

\newtheorem{thm}{Theorem}[section]
\newtheorem{prop}[thm]{Proposition}
\newtheorem{lem}[thm]{Lemma}

\newtheorem{cor}[thm]{Corollary}
\newtheorem{defn}[thm]{Definition}
\newtheorem{rmk}[thm]{Remark}
\newtheorem{ques}[thm]{Question}
\newtheorem{ex}[thm]{Example}

\begin{document}

\title[PI AS-regular algebras of dimension 3 are UFRs]{PI Artin--Schelter regular algebras of dimension 3 are unique factorization rings}

\author{Silu Liu}
\address{School of Mathematical Sciences, Fudan University, Shanghai 200433, China}
\email{liusl20@fudan.edu.cn}

\author{Quanshui Wu}
\address{School of Mathematical Sciences, Fudan University, Shanghai 200433, China}
\email{qswu@fudan.edu.cn}

\thanks{This research is supported by the National Science Foundation of China (Grant No. 12471032).}

\begin{abstract} We prove that all noetherian PI Artin--Schelter regular algebras of dimension $3$ are unique factorization rings. In a certain sense, this result is a noncommutative analogue to the fact that regular local rings of dimension 3 are UFDs. The fact constitutes a crucial component in the proof of the assertion that all regular local rings are UFDs, known as the Auslander--Buchsbaum theorem.
\end{abstract}
\maketitle

\section{Introduction}
Various attempt has been made to extend the concept of unique factorization domains (for short, UFDs) for commutative rings to  noncommutative rings. Recall that a commutative noetherian domain is a UFD if and only if every height-one prime ideal is principal. An early attempt in this direction was made by Chatters in \cite{Cha84} where a  noetherian domain (not necessarily commutative) $R$ is called a \textit{unique factorization domain} if $R$ has at least one height-one prime ideal, and each height-one prime ideal of $R$ is a completely prime principal ideal. 
Such kind of noncommutative rings are still called UFDs in literature. 

Here are some examples of noncommutative UFDs. Group algebras of some polycyclic-by-finite groups \cite{Bro85,CC91,Cha95}; universal enveloping algebras of finite dimensional complex Lie algebras which are either solvable or semisimple \cite{Cha84}; some Iwasawa algebras \cite{Ven03}; some quantum algebras including generic quantum matrices and some quantized enveloping algebras \cite{LLR06}. It is still open whether any universal enveloping algebras of finite-dimensional complex Lie algebra is a UFD.

Noncommutative UFDs, like  commutative UFDs, hold a significant importance in ring theory and other fields. For instance, in recent studies of quantum cluster algebras, Goodearl and Yakimov succeeded in constructing initial clusters for a general family of noncommutative algebras by employing the noncommutative UFD property \cite{GY16, GY17, GY20}.

Moreover, by identifying a normal element that generates a height-one prime ideal as a \textit{prime} element, there is indeed a ``factorization" of the elements in noncommutative UFDs. As a matter of fact, a noetherian domain $R$ is a UFD if and only if that any element in $R$ can be expressed as a product of prime elements and an element without prime factors \cite[Proposition 2.1]{Cha84}. There are also some relevant research concerning (the length of the) factorization of elements in noncommutative noetherian rings \cite{BHL17,BC19,BBNS23}.

However, unlike in the commutative case, some polynomial rings over noncommutative UFDs are non-UFDs. The main reason is that although the requirement of height-one prime ideals being principal are often satisfied by polynomial extensions, completely primeness may not be preserved. Therefore, Chatters and Jordan \cite{CJ86} proposed a notion of \textit{noetherian unique factorization rings} (or briefly, \textit{noetherian UFR}s) for noetherian prime rings. A noetherian prime ring is called a noetherian UFR if  every nonzero prime ideal of it contains a nonzero principal prime ideal (see Definition \ref{def UFR}). Noetherian UFRs are closed under  polynomial extensions and matrix extensions. Some new examples appear as noetherian UFRs, including trace rings of generic matrix rings \cite{LeB86}, and some (generalized) down-up algebras \cite{Jor00,LL13}. 

Regular local rings are an important class of commutative rings, and play a key role in commutative ring theory and algebraic geometry. One of the most surprising results concerning regular local rings is 
that all regular local rings are UFDs, known as the Auslander-Buchsbaum theorem. As a matter of fact, by using dimension induction, Nagata \cite[Proposition 11]{Nag58} first demonstrated that if all regular local rings of dimension 3 are UFDs, then every regular local ring is a UFD. Subsequently, Auslander and Buchsbaum concluded the proof by proving that every regular local ring of dimension $3$ is indeed a UFD utilizing the Auslander-Buchsbaum formula \cite[Theorem 5]{AB59}.

Noetherian Artin--Schelter regular algebras share similar properties with regular local rings in some sense.
An $\mathbb{N}$-graded $\kk$-algebra $A$ is called \textit{connected graded} if $A_0=\kk$. A connected graded $\kk$-algebra $A$ is called \textit{Artin--Schelter regular} (or briefly, AS-regular) of dimension $d$ if 
\begin{enumerate}
    \item $\gld(A)=d<\infty$;
    \item $\Ext_A^i(\kk,A)\cong
        \begin{cases}
            0,\ i\neq d;\\
            \kk, \ i=d.
        \end{cases} $
\end{enumerate}

Note that any noetherian connected graded $\kk$-algebra of finite global dimension has finite GK-dimension \cite[Theorem 0.3]{SZ97}. AS-regular algebras are widely accepted as a noncommutative analogue of polynomial algebras. There are  noetherian  AS-regular algebras that are not UFDs (see Example \ref{ex non-UFD}). Therefore, the wildest expectation is that all noetherian AS-regular algebras are UFRs.
 
In \cite{BH06} Braun and Hajarnavis endeavored to establish a noncommutative analogue of the Auslander--Buchsbaum theorem for some \textit{PI rings}, which are rings satisfying polynomial identities. We refer to \cite[Chapter 13]{MR01} for the material about PI rings. In particular, some noetherian PI AS-regular algebras were involved in \cite{BH06}. Among other results, they calculated height-one prime ideals of two specific noetherian PI AS-regular algebras of dimension $3$, all of which turned out to be principal. Therefore, Braun and Hajarnavis \cite{BH06} conjectured that this might not be occasional. This motivated us to prove the following result. 

\begin{thm}(Theorem \ref{thm 3-dim UFR})\label{thm UFR}
    All noetherian PI AS-regular algebras of dimension (no more than) $3$ are UFRs.
\end{thm}

Theorem \ref{thm UFR} is proved using a homological approach, as is the proof of Auslander--Buchsbaum theorem in \cite{AB59}. While the localization technique is typically not applicable in the noncommutative (graded) context, several noncommutative versions of the class group can be defined.
By applying a noncommutative version of class group (see Section \ref{sec pre}), we prove that it suffices to concentrate on height-one graded prime ideals. See Proposition \ref{prop graded} and Corollary \ref{cor graded} for a more general statement.

\begin{prop}
    Let $A$ be a noetherian PI AS-regular algebra. Then the normal class group of $A$ is isomorphic to the graded normal class group of $A$, namely $\ncl(A)\cong \ncl^{gr}(A)$. Consequently, if  height-one graded prime ideals of $A$ are all projective modules over $A$, then $A$ is a noetherian UFR.
\end{prop}

The projectivity of  height-one graded prime ideals of $A$ follows from the following lemma.
\begin{lem}(Corollary \ref{cor proj})
    Let $A$ be a connected graded $\kk$-algebra, $M$ be a finitely generated graded $A$-module. Then the following are equivalent.
    \begin{enumerate}
        \item $M$ is a projective $A$-module.
        \item $\pd_A(M)\leqslant 1$, and $\Ext_A^1(M,N)=0$ for some nonzero finitely generated graded $A$-module $N$.
    \end{enumerate}
\end{lem}

To be precise, we apply the above lemma to $M=N=P$, where $P$ is a  height-one graded prime ideal of $A$. Some facts are then applied, including that noetherian PI AS-regular algebras are homologically homogeneous rings in the sense of \cite{BH84},\cite{SZ94} and \cite{SV08}. As a consequence, we are able to calculate Ext-groups through local cohomology groups of $\Hom_A(P,P)$ (as a module over the center of $A$) with the aid of Lemma \ref{lem ext-lc}, which eventually implies that $\Ext_A^1(P,P)=0$. On the other hand, the reflexivity of $P$ combined with the fact that $\gld(A)\leqslant 3$ implies that $\pd_A(P)\leqslant 1$.

It is worth mentioning again that Nagata's dimension induction \cite[Proposition 11]{Nag58}, which relies on the localization technique of local rings, is unavailable for connected graded algebras. It is unknown whether Theorem \ref{thm UFR} holds for noetherian PI AS-regular algebras of any dimensions.

It has been proved in literature that some non-PI AS-regular algebras are also noetherian UFRs. For example, the \textit{graded down up algebra}, defined as 
\[A(\alpha,\beta):=\kk\langle x,y\rangle/(x^2y-\alpha xyx-\beta yx^2, xy^2-\alpha yxy-\beta y^2x),\]
where $\alpha,\beta\in\kk$ with $\beta\neq 0$, are noetherian AS-regular algebras of dimension $3$. Through a detailed calculation, Jordan proved in \cite[Section 6]{Jor00} that all height-one prime ideals of graded down up algebras are principal.

The \textit{skew polynomial ring} $S_{\mathcal{P}}:=\kk_{\mathcal{P}}[x_1,\cdots, x_n]$ is the $\kk$-algebra generated by ${x_1,\cdots, x_n}$ subject to the relations $x_j x_i=p_{ij}x_i x_j$, $\forall 1\leqslant i,j \leqslant n$, where $\mathcal{P}=(p_{ij})\in\M_n(\kk^{\times})$ is a matrix where $p_{ij}\, p_{ji}=1$ for all $1\leqslant i,j\leqslant n$. Skew polynomial rings $S_{\mathcal{P}}$ are noetherian AS-regular algebras of dimension $n$. Moreover, they are in fact domains. Skew polynomial rings in which each $p_{ij}$ is not a root of unity can be viewed as so-called CGL extensions, which are UFRs \cite[Theorem 3.6]{LLR06}. By viewing skew polynomial rings as iterated Ore extensions over such skew polynomial rings, one can deduce from \cite[Corollary 4.7]{CJ86} that all skew polynomial rings are noetherian UFRs. 

\begin{prop}(Proposition \ref{prop skew poly})
    All skew polynomial rings are UFRs.
\end{prop}

It is unknown whether AS-regular algebras must be prime rings. It is conjectured that noetherian AS-regular algebras are always domains\cite{ATV91}, which has been open for more than thirty years. It has been proved that all noetherian AS-regular algebras of dimension $3$ are domains \cite{ATV91,Ste96} and that all noetherian PI AS-regular algebras are domains \cite{SZ94}. The natural questions we would like to ask are
\begin{ques}\hspace*{\fill}
    \begin{enumerate}
        \item Are noetherian PI AS-regular algebras always UFRs?
        \item Are noetherian AS-regular algebras of dimension $3$ always UFRs?
    \end{enumerate}
\end{ques}

\section{Preliminaries}\label{sec pre}
Throughout $\kk$ is a fixed base field. Rings are assumed to be unital and associative. A noetherian ring means it is both a left and a right noetherian ring. All modules considered are left modules. For a ring $A$, right $A$-modules are identified with $A^o$-modules, where $A^o$ refers to the opposite ring of $A$. When a $G$-graded ring is mentioned, $G$ always refers to a ordered torsion-free abelian group (for instance,  $\mathbb{Z}^n$). 

Let $A$ be a ring. An element $a$ of $A$ is called a \textit{normal element} if $Aa=aA$. An ideal $I$ of $A$ is called a \textit{principal ideal} if it is an ideal generated by a normal element.

\begin{defn}\cite[Definition in \S 2]{CJ86}\label{def UFR}
    A ring $A$ is called a \textit{noetherian unique factorization ring} (noetherian UFR for short) if $A$ is a  noetherian prime ring such that every nonzero prime ideal of $A$ contains a nonzero principal prime ideal.
\end{defn}

A commutative noetherian domain is a noetherian UFR
in the above sense if and only if it is a UFD in the classical sense \cite[Theorem 5]{Kap74}.

\begin{rmk}\label{rmk principal}\hspace*{\fill}
    \begin{enumerate}
        \item The concept of UFR was latter extended to non-noetherian rings, see \cite[Section 4.2]{Sme16} for more details and generalizations. In this paper, we only discuss the unique factorization property of noetherian rings.
        \item In any noetherian prime ring $A$ (then every left or right regular element in $A$ is regular), an ideal $I$ of $A$ is principal if and only if there exist $a,b \in A$  such that $I=aA=Ab$ \cite[Remark 2]{CJ86}.  Hence, an ideal in a noetherian prime ring is principal as long as it is generated by one element both as a left and a right ideal.
        \item For any noetherian prime ring $A$ satisfying the descending chain condition of prime ideals (for instance, $A$ is a PI ring \cite[Proposition 13.7.15]{MR01}, or $A$ is a $\kk$-algebra with finite GK-dimension \cite[Corollary 3.16]{KL00}), $A$ is a noetherian  UFR if and only if all prime ideals of height one are  principal. 
    \end{enumerate}  
\end{rmk}

Let $X^1(A)=\{P\in\Spec A\mid \height P=1\}$,  the set of prime ideals of $A$ of height one. Recall that for a Krull domain $R$, the divisor class group $\cl(R)$ is generated by prime ideals in $X^1(R)$ \cite[Chapter II]{Fos73}. Similar fact holds for a maximal order over a Krull domain. 

\begin{defn}\cite[Definition 5.3.6]{MR01} \label{def maximal order}
    \begin{enumerate}
        \item Let $R$ be a commutative integral domain with quotient field $K$, and $Q$ be a central simple $K$-algebra. A subring $A$ of $Q$ is called an \textit{$R$-order} in $Q$ if $R \subseteq A$,  $KA=Q$, and $A$ is c-integral over $R$ (that is, for each element $a$ of $A$, $R[a]$ is contained in some finitely generated $R$-module).
        \item An $R$-order $A$ is called a \textit{maximal $R$-order} if it is not properly contained in any other $R$-orders in $Q$. 
    \end{enumerate}   
\end{defn}

\begin{rmk}\hspace*{\fill}\label{rmk different orders}
    \begin{enumerate}
        \item Definition \ref{def maximal order} is slightly different from the one given in \cite[$\S$ II.4]{LVV88}, where an $R$-order is required to be integral over $R$. By \cite[Lemma 5.3.2]{MR01} there is no difference between them when $R$ is a noetherian ring.
        \item The concept of orders can be discussed in a more general setting. Let $Q$ be an artinian simple ring. A subring $A$ of $Q$ is called a \textit{order in $Q$} if $Q$ is a two-sided classical quotient ring of $A$. Let $A$ and $A'$ be orders in $Q$. They are said to be \textit{equivalent} if there exist units $a, b, c, d \in Q$ such that $aAb\subseteq A'$ and $cA'd \subseteq A$. A \textit{maximal order} is an order in $Q$ that is maximal within its equivalent class. According to \cite[Theorem 5.3.13]{MR01}, for an $R$-order $A$ in $Q$ defined in Definition \ref{def maximal order} with $R$ being the center of $A$, $A$ is a maximal $R$-order in $Q$ if and only if $A$ is a maximal order in $Q$. 
    \end{enumerate}
\end{rmk}

\begin{ex}
    Suppose that $A$ is a prime PI ring with center $Z$.  Let $\set=Z\setminus\{0\}$. By Posner's theorem (see \cite[Theorem 13.6.5]{MR01} for example), $Q:=A_{\set}$ is a central simple algebra over $K:=Z_{\set}$. Moreover, it follows from \cite[Theorem 13.6.10]{MR01} that there exists a finitely generated free $R$-module $M$ with $A\subseteq M\subseteq Q$. Therefore, $A$ is a $Z$-order in $Q$. In fact, $A$ is a prime PI ring if and only if that $A$ is a $Z$-order in a central simple algebra \cite[Proposition 5.3.10]{MR01}. 
\end{ex}

\begin{defn}\cite[Definition 3.1.11]{MR01}
    Let $A$ be an $R$-order in $Q$. An $A$-$A$ submodule $I$ of $Q$ is called a \textit{fractional $A$-ideal} if $I$ contains a unit of $Q$, and there exist units $u,v\in Q$ such that $uI\subseteq A$ and $Iv\subseteq A$.
\end{defn}

Let $A$ be an $R$-order in $Q$, $I$ and $J$ be two fractional $A$-ideals. Let
$$(I:_lJ):=\{q\in Q\mid qJ\subseteq I\} \textrm{ and } (I:_rJ):=\{q\in Q\mid Jq\subseteq I\}.$$
When $(I:_lJ)=(I:_rJ)$, it is denoted by $(I:J)$.

If $A$ is a maximal $R$-order in $Q$, then $(I:I)=A$ and $(A:I)=\{q\in Q\mid IqI\subseteq I\}$ for any fractional $A$-ideal $I$ \cite[Proposition 5.1.8]{MR01}. It follows that $(A:I)$ is also a fractional $A$-ideal. A fractional $A$-ideal $I$ is called a \textit{reflexive fractional $A$-ideal} if $(A:(A:I))=I$. 

An $A$-module $M$ is called a \textit{reflexive $A$-module} if the canonical evaluation map $$M\to \Hom_{A^o}(\Hom_A(M,A),A), \ x \mapsto x^{**}$$
is an isomorphism, where $x^{**}(f)=f(x)$ for all $f\in \Hom_A(M,A)$.

\begin{rmk}
    By \cite[Proposition 3.1.15]{MR01} 
    $$\ (I:_rJ)\cong \Hom_A(J,I)  \textrm{ and } (I:_lJ)\cong \Hom_{A^o}(J,I).$$
    Therefore, if $A$ is a maximal $R$-order then a fractional $A$-ideal $I$ is reflexive if and only if $I$ is reflexive as an $A$-module.
\end{rmk}

\begin{thm}\cite[Theorem 2.3]{Sil68}\label{thm div group}
    Let $A$ be a maximal $R$-order, where $R$ is a Krull domain. The set of all reflexive fractional $A$-ideals is a group with the product given by $I* J:=(A:(A:IJ))$, which is denoted by $\di(A)$. Moreover $\di(A)$ is a free abelian group generated by $X^1(A)$. 
\end{thm}

\begin{rmk} A more general form of
    Theorem \ref{thm div group} is proved for tame $R$-orders in \cite[Proposition 2.3]{Fos68} and \cite[Theorem II.4.9]{LVV88}. 
    In fact,  by \cite[Proposition 1.3]{Fos68} and \cite[Theorem 21.4]{Rei03}, if $R$ is a Krull domain, maximal $R$-orders are automatically tame $R$-orders. 
\end{rmk}

\begin{rmk}
    Let $A$ and $R$ be as in Theorem \ref{thm div group}. For a fractional $A$-ideal $I$, let $I^{-1}=(A:I)$. Then $I^{-1}*I=I*I^{-1}=A$.
\end{rmk}

Let $\npr(A)=\{Ax\mid 0\neq x\in Q, Ax=xA\}$. It is direct to verify that $\npr(A)$ is a subgroup of $\di(A)$. Similarly, let $\cpr(A)=\{Ak\mid 0\neq k\in K\}$.

The \textit{normal class group} of $A$ is defined as $\ncl(A)=\di(A)/\npr(A)$; the \textit{central class group} of $A$ is  defined as $\ccl(A)=\di(A)/\cpr(A)$. Obviously, $\cpr(A)$ is a subgroup of $\npr(A)$, and consequently $\ncl(A)$ is a quotient group of $\ccl(A)$.

Throughout this paper, the following hypothesis is frequently assumed.
\begin{equation}\label{hypo}\tag{\textbf{H}}  
    \begin{split}
        & \quad \textit{$A$ is a prime PI ring with a noetherian normal center $Z$ such that }\\
        & \textit{$A$ is a maximal $Z$-order}.
    \end{split}
\end{equation}

It is a noteworthy that if $A$ satisfies Hypothesis \eqref{hypo} then $A$ is a finitely generated module over $Z$, hence also a noetherian ring\cite[Corollary 13.6.14]{MR01}.

As in the commutative case, $\ncl(A)$ measures the lacking of unique factorization property of $A$. Keeping Remark \ref{rmk different orders} in mind, the following theorem follows immediately from \cite[Theorem 4.6]{Aka09}, as any noetherian PI prime ring has enough invertible ideals.

\begin{thm}\label{Aka's result} 
    Suppose that $A$ satisfies Hypothesis \eqref{hypo}. Then $A$ is a noetherian UFR if and only if $\ncl(A)=\{1\}$.
\end{thm}

Let $A$ be a $G$-graded ring. Then $X_{gr}^1(A)$, $\di^{gr}(A)$, $\npr^{gr}(A)$, $\cpr^{gr}(A)$, $\ncl^{gr}(A)$ and $\ccl^{gr}(A)$ can be defined similarly by considering the corresponding graded objects.

\section{Unique factorization property of graded algebras}
In this section we prove that for some connected graded $\kk$-algebra $A$, $X^1_{gr}(A)$ controls the unique factorization property of $A$. Class groups are used to measure the lack of unique
factorization property of Krull domains in commutative algebra and number theory. In the noncommutative context, normal class groups or central class groups are similarly used, including \cite{LeB84,LeB84_2,LV86,LVV88, Aka09}. We first prove that for any graded domain satisfying Hypothesis \eqref{hypo} its normal class group is isomorphic to its graded normal class group (Proposition \ref{prop graded}). Then we prove that such kind of connected graded algebra is a noetherian UFR if and only if all of its height-one graded prime ideals are left and right projective (Corollary \ref{cor graded}), which is a key step to prove the main result Theorem \ref{thm 3-dim UFR} in this paper.

\begin{prop}\label{prop graded}
    Let $A$ be a $G$-graded ring which is also a domain. If $A$ satisfies Hypothesis \eqref{hypo}, then $\ncl(A)\cong \ncl^{gr}(A)$.
\end{prop}
\begin{proof}
    Let $\set$ be a multiplicatively closed subset of $Z\setminus \{0\}$. Then $A_{\set}$ is a $Z_{\set}$-order in $Q$ such that the quotient field $K$ of $Z$ is also the quotient field of $Z_{\set}$. Moreover,  $A_{\set}$ is a maximal $Z_{\set}$-order in $Q$ \cite[Corollary II.4.12]{LVV88}.
 
    Obviously, $\Spec(A_{\set})=\{P_{\set}\mid P\in\Spec A, P\cap \set=\varnothing\}$. Therefore, the homomorphism $\phi: \di(A)\to \di(A_{\set}), I \mapsto I_{\set}$ is surjective, as $\phi(P) = P_{\set}$ for all $P\in X^1(A)$. Obviously, $\phi(\npr(A)) \subseteq \npr(A_{\set})$. As a consequence, $\phi$ induces a surjection 
    $$\bar{\phi}: \ncl(A)\to \ncl(A_{\set}).$$ 
    
    Now, let $\set$ be the set of homogeneous elements of $Z\setminus\{0\}$. We claim that $$\phi(\npr(A))= \npr(A_{\set}).$$

    Suppose $x\in Q$ such that $xA_{\set}=A_{\set}x$. Then $x$ is regular, and $x$ induces an algebra automorphism $\eta_x: A_{\set} \to A_{\set}$ given by $xa=\eta_x(a)x$. We may assume that $x\in A$ (otherwise, $rx\in A$ for some $r\in S$, we replace $x$ by $rx$). Let $x=\sum_{i=1}^n x_i$ where $x_1,\cdots, x_n$ are homogeneous element of $A$ such that $\deg(x_1)<\cdots<\deg(x_n)$ as $A$ is $G$-graded. Note that $A_{\set}$ is also a $G$-graded ring with $\deg(a/s):=\deg(a) \deg(s)^{-1}$ for any homogeneous element $a \in A$ and $s \in \set$. For any homogeneous element $a\in A_{\set}$, it follows from the equality
    $$(x_1+\cdots+x_n)a=\eta_x(a)(x_1+\cdots +x_n)$$
    and $A$ is a domain that $\eta_x(a)$ is a homogeneous element in $A_{\set}$ of degree $\deg(a)$. Moreover, $x_i a=\eta_x(a)x_i$ for any $a\in A_{\set}$ and $1\leqslant i\leqslant n$. In particular, $x_1$ is a homogeneous element of $A$ such that $x_1A_{\set}=A_{\set}x_1$ and  $\eta_{x_1}=\eta_x$. Since $A$ is a $G$-graded domain, by the graded version of Posner's theorem \cite[Theorem C.I.2.8]{NV82}, $A_{\set}$ is a $G$-graded division ring, that is, each homogeneous element is invertible. It follows that $x_1^{-1}$ is a homogeneous element of $A_{\set}$ such that $x_1^{-1}A_{\set}=A_{\set}x_1^{-1}$, and $\eta_{x_1^{-1}}=\eta_{x_1}^{-1}$. 
    Consequently, for any $a\in A_{\set}$,  
    $$x_1^{-1}xa=x_1^{-1}\eta_{x}(a)x=\eta_{x_1^{-1}}(\eta_x(a))x^{-1}x=\eta_{x_1}^{-1}(\eta_x(a))x^{-1}x = ax_1^{-1}x.$$
    Therefore, $x_1^{-1}xA=Ax_1^{-1}x$. In other words,  $Ax_1^{-1}x\in \npr(A)$. Hence $A_{\set}x=A_{\set}x_1x_1^{-1}x=A_{\set}x_1^{-1}x=\phi(Ax_1^{-1}x)$, which proves the claim.

    Then we have the following exact commutative diagram.
    \begin{center}
    \begin{tikzcd}[column sep = small, row sep = 1.5em]

    & 1\arrow[d]   &1\arrow[d]   &1\arrow[d]\\
1 \arrow[r] & \ker\phi|_{\npr(A)} \arrow[r] \arrow[d] & \ker\phi \arrow[r] \arrow[d] & \ker\bar{\phi}\arrow[r]\arrow[d]& 1\\

1 \arrow[r] & \npr(A) \arrow[r] \arrow[d, "\phi|_{\npr(A)}"] & \di(A) \arrow[r] \arrow[d, "\phi"] & \ncl(A) \arrow[r] \arrow[d, "\bar{\phi}"] & 1 \\
1 \arrow[r] & \npr(A_{\set}) \arrow[r] \arrow[d]             & \di(A_{\set}) \arrow[r] \arrow[d]  & \ncl(A_{\set}) \arrow[r] \arrow[d]        & 1 \\
            & 1                                              & 1                                  & 1                                         &  
\end{tikzcd}
\end{center}
Hence  
$\ker\bar{\phi}\cong \ker\phi/\ker\phi|_{\npr(A)}.$

Now, suppose $P\in X^1(A)$. Consider the ideal $P^g$ of $A$ generated by the homogeneous elements of $P$, which is also a prime ideal of $A$. It follows that either $P^g=P$ or $P^g=0$. Therefore, $\ker\phi$ is generated by $$\{P\in X^1(A)\mid P\cap\set\neq\varnothing\}= X_{gr}^1(A),$$ which implies that $\ker\phi=\di^{gr}(A)$. Note $\di^{gr}(A)\cap \npr(A)=\npr^{gr}(A)$. Therefore,
    $$\ker\phi/\ker\phi|_{\npr(A)}=\ker\phi/\ker\phi\cap\npr(A)=\di^{gr}(A)/\npr^{gr}(A).$$
It follows that  $\ker\bar{\phi}\cong \ncl^{gr}(A)$.

On the other hand, by \cite[Proposition I.1.1]{LVV88}, $Z_{\set}$ is a commutative UFD. Since $A_{\set}$ is an Azumaya algebra over $Z_{\set}$ \cite[Proposition II.3.4]{LVV88}, it follows from \cite[Theorem II.4.24]{LVV88} that $\ccl(A_{\set}) \cong \cl(Z_{\set})\cong \{1\}$. So $\ncl(A_{\set})=\{1\}$, which implies that $\ncl(A)\cong \ker\bar\phi$. 

It follows that $\ncl(A)\cong \ncl^{gr}(A)$.
\end{proof}

\begin{cor} \label{cor graded}
    Suppose that $A$ is a connected graded $\kk$-algebra satisfying Hypothesis \eqref{hypo}, and $A$ is a domain. Then $A$ is a noetherian UFR if and only if every $P\in X_{gr}^1(A)$ is projective both as a left and a right $A$-module.
\end{cor}

\begin{proof}
    By Theorem \ref{Aka's result} and Proposition \ref{prop graded}, it suffices to prove that any $P\in X_{gr}^1(A)$ such that both $P_A$ and $_AP$ are projective is a principal ideal.
     
    Since $A$ is connected graded, any finitely generated graded projective $A$-module is free of finite rank. According to Remark \ref{rmk principal} (2), it suffices to show that $P$ is of rank $1$ on both sides. This follows from the following general fact.

    Any left ideal $I$ in a noetherian prime ring $B$ such that ${}_B I$ is free is free of rank $1$. Note that the free generators of ${}_B I$  must be regular elements of $B$. If $a,b\in I$ are regular elements of $B$, then $a,b$ are invertible elements in $Q$, where $Q$ is the left classical quotient ring of $B$. Hence $Qa=Qb=Q$. As a result $Ba\cap Bb\neq 0$. 
\end{proof}

Let $A$ be a ring, $M$ be an $A$-$A$ bimodule. For any $\sigma,\tau\in \Aut(A)$, ${}^{\sigma}M^{\tau}$ is the $A$-$A$ bimodule with the same additive group as $M$, and  the new action $a\cdot x\cdot b:=\sigma(a)x\tau(b)$ for any $x \in M$, $a ,b\in A$. If $A$ is moreover a connected graded $\kk$-algebra, then any graded invertible $A$-$A$ bimodule must be of form ${}^1A^{\sigma}$ for some $\sigma\in\Aut_{\kk}(A)$ \cite[Lemma 1.7]{RRZ14}. Invertible $A$-$A$ bimodules (which are not graded) usually are more complicated. The central symmetric invertible $A$-$A$ bimodules over a connected graded $\kk$-algebra satisfying Hypothesis \eqref{hypo} are in fact quite clear (see Proposition \ref{prop pic-oz}). 

We consider those invertible $A$-$A$ bimodules $U$ which are \textit{symmetric} over $Z$, that is, $zx=xz$ for all $z\in Z$ and $x\in U$. 
The isomorphic classes of invertible $A$-$A$ bimodules which are symmetric over $Z$ constitute a subgroup of the Picard group $\Pic(A)$. This subgroup is called the \textit{central Picard group} of $A$, denoted by $\Picent(A)$. 

Chan, Gaddis, Won and Zhang \cite{CGWZ25} studied the \textit{ozone group} of $A$, which consists of all the automorphisms of $A$ that fix the center $Z$. Namely,  
\[\Oz(A):=\{\sigma\in\Aut_{\kk}(A)\mid \sigma(z)=z,\forall z\in Z\}.\]
For any  $\sigma\in\Oz(A)$,  ${}^1A^{\sigma}$ is an invertible $A$-bimodule symmetric over $Z$. Therefore, there is a map $\Oz(A)\to \Picent(A), \sigma \mapsto [{}^1A^{\sigma}]$. If $A$ is a connected graded $\kk$-algebra which is a domain, then there are no nontrivial invertible elements in $A$, and consequently there are no nontrivial inner automorphisms of $A$, which implies that the above map is injective. We claim that it is also surjective.

\begin{prop}\label{prop pic-oz}
    Let $A$ be a connected graded $\kk$-algebra satisfying Hypothesis \eqref{hypo}. Then $\Picent(A)\cong \Oz(A)$.
\end{prop}
\begin{proof}   
    It suffices to prove that the map  $\Oz(A)\to \Picent(A), \sigma \mapsto [{}^1A^{\sigma}]$ is surjective, that is, for any $[U]\in \Picent A$, there exists some $\sigma\in\Oz(A)$ such that $U\cong {}^1 A^{\sigma}$ as $A$-$A$ bimodules.

    By \cite[Lemma 3.2]{GM94}, there exists an ideal $I$ of $A$ such that $U\cong I$ as $A$-$A$ bimodules. Since $I$ is a finitely generated projective $A$-module, it is a reflexive fractional $A$-ideal. Note that two reflexive fractional $A$-ideals $I$ and $I'$ are isomorphic $A$-$A$ bimodules if and only if there exists an invertible elements $a, b \in Q$ such that $xa=bx\in I'$ for any $x \in I$.
    Consequently, $\Picent(A)$ is a subgroup of $\ccl(A)$.

    Let $\set$ be a multiplicatively closed subset of $Z\setminus \{0\}$. By a similar argument as in Proposition \ref{prop graded}, the surjective map  $\phi:\di(A)\to \di(A_{\set})$, $I\mapsto I_{\set}$ induces a surjective map $\bar\phi:\ccl(A)\to\ccl(A_{\set})$, which further induces a surjective map $\tilde{\phi}: \Picent(A)\to \Picent(A_{\set})$.
    
    In particular, let $\set$ be the set of homogeneous elements of $Z\setminus \{0\}$. Again by \cite[Proposition I.1.1]{LVV88} $\Picent(A_{\set})\subseteq \ccl(A_{\set})=\{1\}$,  
    which implies that $\ccl(A)\cong\ker\bar{\phi}=\ccl^{gr}(A)$, and that $\Picent(A)\cong \ker\tilde{\phi}=\ker\bar\phi\cap\Picent(A)\subseteq\ccl^{gr}(A)$.

    Therefore, we may assume that  $I$ is a graded ideal of $A$ which is moreover an invertible $A$-$A$ bimodule. Since $A$ is connected graded, $U\cong I\cong {}^1A^{\sigma}$ for some $\sigma\in\Aut_{\kk}(A)$ (see, for example, \cite[Lemma 1.7]{RRZ14}). It follows from $[U]\in\Picent(A)$ that $\sigma\in\Oz(A)$.
\end{proof}
\begin{rmk}
    If $R$ is commutative ring, then $\Picent(R)=\Pic(R)$ is the Picard group of $R$. Note that noetherian normal domains are automatically maximal orders over themselves. Therefore, Proposition \ref{prop pic-oz} can be viewed as a generalization of \cite[Lemma 5.1]{MS76}, which states that connected graded noetherian normal domains have trivial Picard groups.
\end{rmk}

\section{Some noetherian PI AS-regular algebras are UFRs}
In this section, we prove the main result in this paper that any noetherian PI AS-regular algebra of dimension $3$ is a noetherian UFR (see Theorem \ref{thm 3-dim UFR}). The proof relies on Propositions \ref{prop pi AS-regular} and \ref{prop hom hom}, Corollary \ref{cor proj} and Lemma \ref{lem ext-lc}.  

In general, a PI ring may not be a finitely generated module over its center. Noetherian PI AS-regular algebras, however, are finitely generated modules over their centers. In fact, noetherian PI AS-regular algebras satisfy Hypothesis \eqref{hypo}.

\begin{prop}\label{prop pi AS-regular}\cite[Corollary 1.2]{SZ94}
    Let $A$ be a noetherian PI AS-regular algebra with center $Z$. Then $A$ is a domain, and $A$ is a maximal $Z$-order. Moreover, $A$ is a finitely generated module over $Z$, which is a noetherian normal domain.
\end{prop}

We recall the notion of homologically homogeneous rings \cite{BH84}. We adopt the description in \cite{SV08} for convenience, which coincide with \cite{BH84} and \cite{SZ94} for noetherian prime affine $\kk$-algebras.
\begin{defn}\cite{SV08}
    Let $A$ be a affine prime $\kk$-algebra that is a finitely generated module over its center $Z$. We say $A$ is a \textit{homologically homogeneous ring of dimension $d$}, if all simple $A$-modules have the same projective dimension $d$.
\end{defn}

The following facts are well-known in literature. For the material concerning  dualizing complexes we refer to \cite{vdB97, YZ99}. Note that noetherian connected graded $\kk$-algebras are automatically affine $\kk$-algebras.

\begin{prop}\label{prop hom hom}
    Let $A$ be a noetherian PI AS-regular algebra of dimension $d$. Then the following hold.
    \begin{enumerate}
        \item $A$ is a homologically homogeneous ring of dimension $d$.
        \item $\gld(A_{\p})=\height\p$ for any $\p\in\Spec Z$.
    \end{enumerate} 
\end{prop}
\begin{proof}
    (1) This can be deduced from \cite{SZ94}. Or directly, by \cite[Proposition 6.18]{YZ99}, ${}^1 A^{\sigma}[d]$ is a rigid dualizing complex of $A$, where $\sigma\in\Aut_{\kk}(A)$. The rest follows from \cite[Proposition 2.9]{SV08} as ${}^1 A^{\sigma}$ is an invertible $A$-$A$ bimodule.


    (2) This is just \cite[Theorem 3.5]{BH84}. 
\end{proof}

Let $A$ be a connected graded $\kk$-algebra, $\m_A:=A_{\geqslant 1}$ be the unique graded maximal ideal of $A$. 
We provide a proof of the subsequent fact for the sake of convenience.

\begin{lem}\label{lem ext-nonvanishing}
    Let $A$ be a connected graded $\kk$-algebra, $M$ be a finitely generated graded $A$-module. If   
    $\pd_A(M)=n<\infty$, then $\Ext_A^n(M,N)\neq0$ for any nonzero finitely generated graded $A$-module $N$.
\end{lem}
\begin{proof}
    It follows from $\pd_A(M)=n$ that $M$ has a minimal free resolution 
    \begin{equation} \label{mini-resolution-M}
     0\to F_n\xrightarrow{d_n} F_{n-1}\to...\to F_0\to 0.   
    \end{equation}
    Then, 
    $\Hom_A(F_{n-1},N)\to \Hom_A(F_n,N)\to \Ext_A^n(M,N)\to 0$ is exact.
    
    Suppose $\Ext_A^n(M,N)=0$. Then every morphism $f\in \Hom_A(F_n,N)$ factors through $d_n$, that is, there exists $\Tilde{f}:F_{n-1}\to N$ such that $f=\Tilde{f}d_n$. The minimality of \eqref{mini-resolution-M} implies that
    \[f(F_n)=\Tilde{f}d_n(F_n)\subseteq \Tilde{f}(\m_AF_{n-1})=\m_A\Tilde{f}(F_{n-1})\subseteq \m_AN .\]
    Since $F_n$ is a free generator in the graded $A$-module category, there exist morphisms $\{f_i:F_n\to N\}_{i\in I}$ such that $N=\Sigma_{i\in I}f_i(F_n)$. As a result, $N\subseteq \m_AN$, which is a contradiction.
\end{proof}

\begin{cor}\label{cor proj}
    Let $A$ be a connected graded $\kk$-algebra, $M$ be a finitely generated graded $A$-module. Then the following are equivalent.
    \begin{enumerate}
        \item $M$ is a projective $A$-module.
        \item $\pd_A(M)\leqslant 1$, and $\Ext_A^1(M,N)=0$ for some nonzero finitely generated graded $A$-module $N$
    \end{enumerate}
\end{cor}

For the basic facts about  noncommutative  local cohomology we refer to \cite{Jor97}.   Let $\Gr A$ be the category of graded $A$-modules.
The \textit{torsion functor} is defined as 
\[\Gamma_{\m_A}(-):=\colim\Hom_A(A/A_{\geqslant n},-):\Gr A\to \Gr A.\]
It is direct to verify that for a finitely generated graded $A$-module $M$, $\Gamma_{\m_A}(M)$ is the largest graded submodule of $M$ of finite length. Moreover $\Gamma_{\m_A}$ is a left exact functor. Let $\hh_{\m_A}^i(-):=R^i\Gamma_{\m_A}(-)$ be the $i$-th derived functor of $\Gamma_{\m_A}$, which will be called the \textit{$i$-th local cohomology functor}. Note that
\[\hh_{\m_A}^i(-)=\colim\Ext_A^i(A/A_{\geqslant n},-):\Gr A\to \Gr A.\]
The \textit{depth} of a graded $A$-module $M$ is defined as
\[\depth_A(M):=\inf\{i\in \mathbb{N}\mid \hh_{\m_A}^i(M)\neq 0\}.\]
By \cite[Proposition 4.3]{Jor97} $\depth_A(M)=\inf\{i\in\mathbb{N}\mid \Ext_A^i(\kk, M)\neq 0\}$. Therefore, for any AS-regular algebra $A$ of dimension $n$, one conclude immediately that $\hh_{\m_A}^i(A)=0$ for all $i\neq n$.

The next lemma is modified from a result in commutative case \cite[Lemma 3.2]{ACS20} and \cite[Lemma 2.2]{Asg23}, which plays a key role in the proof of Theorem \ref{thm 3-dim UFR}.
\begin{lem}\label{lem ext-lc}
    Let $A$ be a noetherian connected graded $\kk$-algebra which is a finitely generated module over its center $Z$. Let $M$ and $N$ be two finitely generated graded $A$-modules, and $t$ be an integer such that $2 \leqslant t \leqslant \depth_Z(N)$. If
  $\Ext_A^i(M,N)$ are of finite length as $Z$-modules for $i=1,\cdots,t-1$,  then the following hold.
    \begin{enumerate}
        \item $\Ext_A^i(M,N)\cong \hh_{\m_Z}^{i+1}(\Hom_A(M,N))$ for $i=1,\cdots,t-2$.
        \item There is an injective $Z$-morphism $\Ext_A^{t-1}(M,N)\hookrightarrow \hh_{\m_Z}^{t}(\Hom_A(M,N))$.
    \end{enumerate}
\end{lem}

\begin{proof}
    Let $F_{\bdot} \to M \to 0$ be a free resolution of ${}_AM$.
    By considering the complex
    \begin{equation*}
    \begin{tikzcd}[column sep=1.6em]
        \dotsb\ar[r,"d^*_{i-1}"] &
        \Hom_A(F_{i-1}, N)\ar[r,"d_i^*"] &
        \Hom_A(F_i, N) \ar[r,"d^*_{i+1}"]&
        \Hom_A(F_{i+1}, N)\ar[r]&\dotsb
    \end{tikzcd}
    \end{equation*}
    we obtain the following exact sequences:
    \begin{equation}\label{ses1}
        0\to \Ext_A^i(M,N)\to T_i\to X_i\to 0
    \end{equation}
    and
    \begin{equation}\label{ses2}
        0\to X_i\to \Hom_A(F_{i+1},N)\to T_{i+1}\to 0,
    \end{equation}
    where $T_i:=\Hom_A(F_{i},N)/\Img(d_i^*)$, and $X_i:=\Img(d_{i+1}^*)$ for all $i\geqslant 0$. 
    
    For any $0<i<t$, it follows from the assumption that $\Ext_A^i(M,N)$ is of finite length over $Z$ that 
    \begin{align*}
        \Ext_A^i(M,N)=& \Gamma_{\m_Z}(\Ext_A^i(M,N));\\
        \hh_{\m_Z}^j(\Ext_A^i(M,N))=& 0, \text{ for all } j>0.
    \end{align*}
    Obviously, for any $i\geqslant 0$, $\depth_Z(\Hom_A(F_{i},N))=\depth_Z(N)\geqslant t\geqslant 2$. Therefore, 
    $$\Gamma_{\m_Z}(X_i)\hookrightarrow \Gamma_{\m_Z}(\Hom_A(F_{i+1},N))=0.$$
    By applying the local cohomology functor to \eqref{ses1} and \eqref{ses2}, we obtain that
    \begin{align}
        \Ext_A^i(M,N) \cong & \, \Gamma_{\m_Z}(T_i),\ \forall \ 0<i<t;\label{equiv1}\\
        \hh_{\m_Z}^j(T_i) \cong & \hh_{\m_Z}^j(X_i),\ \forall \ 0<i<t,\ j>0;\label{equiv2}\\
        \hh_{\m_Z}^j(T_{i+1}) \cong & \hh_{\m_Z}^{j+1}(X_i),\ \forall \ 0 \leqslant j<t-1,\ i\geqslant 0.\label{equiv3}
    \end{align}
    Therefore, for any $0<i<t$,
    \begin{equation}\label{les}
        \Ext_A^i(M,N)\cong \Gamma_{\m_Z}(T_i)\cong \hh_{\m_Z}^1(X_{i-1})\cong \hh_{\m_Z}^1(T_{i-1})\cong  \cdots \cong \hh_{\m_Z}^{i-1}(T_1). 
    \end{equation}   
    By applying the local cohomology functor to the exact sequences
    $$0\to \Hom_A(M,N)\to \Hom_A(F_0,N)\to X_0\to 0 $$
    and
    $$0\to X_0\to \Hom_A(F_1,N)\to T_1\to 0,$$
    we obtain that 
    $$\hh_{\m_Z}^{i-1}(T_1)\cong \hh_{\m_Z}^i(X_0)\cong \hh_{\m_Z}^{i+1}(\Hom_A(M,N)) ,\ \forall \ 0<i<t-1.$$
    There is also an injective morphism 
    $$0\to \hh_{\m_Z}^{t-1}(X_0)\to \hh_{\m_Z}^t(\Hom_A(M,N)).$$
    Combined with \eqref{les} and \eqref{equiv3}, it follows that
    $$\Ext_A^{t-1}(M,N)\cong \hh_{\m_Z}^{t-2}(T_1)\cong \hh_{\m_Z}^{t-1}(X_0)\hookrightarrow \hh_{\m_Z}^{t}(\Hom_A(M,N)).$$ 
\end{proof}

Now we are ready to prove the main result in this paper.
\begin{thm}\label{thm 3-dim UFR}
    All noetherian PI AS-regular algebras of dimension no more than $3$ are noetherian UFRs. 
\end{thm}

\begin{proof}
    By Corollary \ref{cor graded}, it suffices to prove that all the prime ideals in $X^1_{gr}(A)$ are left and right projective $A$-modules.

    Note that height-one prime ideals of $A$ are reflexive $A$-modules. For a finitely generated reflexive $A$-module $M$, we have that \cite[Corollary 2.13]{QWZ19}
    $$\pd_A(M)\leqslant \max\{0,\gld(A)-2\}.$$

    If $\gld(A)<3$ then all reflexive $A$-modules are projective. In particular, all height-one prime ideals of $A$ are projective $A$-modules. Therefore, it is left to prove the case where $\gld(A)=3$. By the above inequality we obtain that $\pd_A(P)\leqslant 1$ for any $P\in X_{gr}^1(A)$.
 
    For any  $\p\neq \m_Z\in\Spec Z$ such that $\p$ is a graded prime ideal, $$\height\p< \height\m_Z=\Kdim(Z)=3.$$ 
    By Proposition \ref{prop hom hom} $\gld(A_{\p})=\height\p\leqslant 2$. Note that for any such $\p$, either $P_{\p}=A_{\p}$ or $P_{\p}\in X^1(A_{\p})$. In any case, $P_{\p}$ is a reflexive $A_{\p}$-module. Again by \cite[Corollary 2.13]{QWZ19}, $P_{\p}$ is a projective $A_{\p}$-module. Consequently 
    $$\Ext_A^1(P,P)_{\p}\cong \Ext_{A_{\p}}^1(P_{\p},P_{\p})=0,$$
    which implies that $\Ext_A^1(P,P)$ is of finite length over $Z$. By Lemma \ref{lem ext-lc} and \cite[Lemma 4.15]{YZ99},  
    \[\Ext_A^1(P,P)\hookrightarrow\hh_{\m_Z}^2(\Hom_A(P,P))\cong \hh_{\m_Z}^2(A)\cong \hh_{\m_A}^2(A)=0.\]
    It follows from $\pd_A(P)  \leqslant 1$ and Corollary \ref{cor proj} that ${}_AP$ is projective. 
\end{proof}

Unique factorization property of noetherian PI AS-regular algebras has several applications. For instance, it guarantees the existence of reflexive hull discriminant over these algebras, which is a powerful tool in determining automorphism groups. We refer to \cite{CGWZ22} for the unexplained terminology.

\begin{cor}
    Let $A$ be a noetherian PI AS-regular algebra of dimension $3$. Then the $\overline{\mathcal{R}}_v^p$-discriminant $\Bar{\varrho}_v^{[p]}(A/Z)$ always exists. In particular, the $\overline{\mathcal{R}}$-discriminant $\Bar{\varrho}(A/Z)$ exists.
\end{cor}

\begin{proof}
    By Theorem \ref{thm 3-dim UFR} and Proposition \ref{prop hom hom}, $A$ is a homologically homogeneous noetherian UFR with affine center $Z$. The conclusion follows from \cite[Theorem 4.8]{CGWZ22}.
\end{proof}

Although all noetherian PI AS-regular algebras are domains, some of them may not be UFDs.

A nonzero non-unit normal element $a$ of a prime noetherian ring $A$ is called \textit{irreducible} if there are no non-unit normal elements $b$ and $c$ of $A$ such that $a=bc$.

\begin{lem} \label{lem normal elements and units}
    Let $a$ be a normal element of a ring $A$. 
    \begin{enumerate}
        \item $a$ is a unit in $A$ if and only if $a$ is left (or right) invertible in $A$.
        \item Let $A$ be a prime ring. If $a=bc$ for $b,c \in A$, where $b$ is a normal element, then $c$ is also a normal element.
    \end{enumerate}
\end{lem}
\begin{proof}
    (1) Suppose that $a$ is left invertible. Then $A=Aa=aA$, which implies that $A$ is also right invertible.

    (2) Since $A$ is a prime ring, normal elements are regular.  It follows from $bcA=aA=Aa=Abc=bAc$ that $cA=Ac$, that is, $c$ is a normal element. 
\end{proof}

\begin{prop} \label{prop irreducible elememts in UFR}
    Let $A$ be a noetherian UFR. For any non-unit nonzero normal element $a \in A$, $Aa$ is a prime ideal if and only if $a$ is irreducible.
\end{prop}

\begin{proof}
    Suppose $Aa$ is a prime ideal. If $a=bc$ for some normal elements $b$ and $c$ of $A$, then $bAc=Abc=Aa$. We may assume that $b\in Aa$. Therefore, there exists $b'\in A$ such that $b=ab'$, which implies that $1=cb'$. Consequently, $c$ is a unit.

    On the other hand, suppose that $a$ is irreducible. Take a prime ideal $P$ of $A$ such that $P$ is minimal over $Aa$. By Jategaonkar's principal ideal theorem  \cite[Theorem 4.1.11]{MR01} $\height P=1$. Since $A$ is a noetherian UFR, there exists a normal element $b$ of $A$ such that $P=Ab$. As a result, there exists $a'\in A$ such that $a=a'b$. It follows from Lemma \ref{lem normal elements and units} that $a'$ is a normal element of $A$. Since $a$ is irreducible, it follows that $a'$ is a unit, and consequently, $Aa=Ab$ is a prime ideal.
\end{proof}

\begin{ex}\label{ex non-UFD}
    Consider the graded algebra $A:=\kk\langle x,y\rangle/(xy+yx)$ with $\deg x =\deg y =1$, where $\ch \kk\neq 2$. Then $A$ is a noetherian PI AS-regular algebra of dimension $2$, and so $A$ is a noetherian UFR by Theorem \ref{thm 3-dim UFR}. Although all height-one prime ideals of $A$ are all principal, some of them may not be completely prime, which means that $A$ is not a UFD.

    Obviously, $a:=x^2+y^2$ is a central element of $A$. We claim that $a$ is an irreducible normal element. 
    Suppose on the contrary that $a=bc$ for some non-unit normal elements $b,c\in A$. Since $a$ is a homogeneous element (of degree $2$) of $A$, which is a domain, we conclude that $b$ and $c$ are both homogeneous elements of degree $1$. Since $A_1=\kk x\oplus \kk y$, there exist $\lambda,\mu\in\kk$ such that $b=\lambda x+\mu y$. Since $b$ is a normal element, a direct calculation implies that either $\lambda=0$ or $\mu=0$. Things are similar for $c$. We conclude immediately that $bc\neq x^2+y^2$, which is a contradiction.

    Therefore, $Aa$ is a prime ideal of height-one by Proposition \ref{prop irreducible elememts in UFR}. Since $a=(x+y)^2$,  $Aa$ is not completely prime. 
\end{ex}

In the process of demonstrating that regular local rings are Unique Factorization Domains (UFDs), Nagata was able to reduce the proof to the three-dimensional case \cite[Proposition 11]{Nag58}, which was ultimately resolved by Auslander and Buchsbaum in \cite{AB59}. However, a comparable dimension reduction does not apply to connected graded algebras. Nonetheless, certain PI AS-regular algebras with higher dimensions remain unique factorization rings (UFRs).

We need the following result from  \cite{CJ86} later. We refer to \cite[Section 4]{CJ86} for the unexplained terminology.

\begin{prop}\label{prop ufr and ore ext}
    Let $A$ be a noetherian UFR. 
    \begin{enumerate}
        \item \cite[Theorem 4.2, Corollary 4.3]{CJ86} Let $\sigma$ be an automorphism of $A$. If every nonzero $\sigma$-prime ideal contains a nonzero principal $\sigma$-ideal, then $A[x;\sigma]$ is a noetherian UFR. In particular, $A[x;\sigma]$ is a noetherian UFR if $\sigma$ is an automorphism of finite order.
        \item \cite[Theorem 5.4]{CJ86} Let $\delta$ be a derivation of $A$ such that every nonzero $\delta$-prime ideal contains a nonzero principal $\delta$-ideal. Then $A[x;\delta]$ is a noetherian UFR. 
    \end{enumerate}
\end{prop}

\begin{prop}\label{prop ore ext of pi}
    Let $B$ be a noetherian PI AS-regular $\kk$-algebra. Suppose there exists an algebra $A$ and $\sigma\in\Aut_{\kk}(A)$ such that $B=A[x;\sigma]$. If $A$ is a noetherian UFR, then so is $B$.
\end{prop}

\begin{proof}
    By definition, $xa=\sigma(a)x$ for any $a \in A$. Let  $$\Tilde{\sigma}:B\to B, x\mapsto x \textrm{ and } a\mapsto \sigma(a),\, \forall a \in A.$$
    Then $\Tilde{\sigma}\in\Aut_{\kk}(B)$. Moreover, $xb=\Tilde{\sigma}(b)x$ for any $b \in B$. Since $B$ is a domain, we conclude that $\Tilde{\sigma}(b)=b$, $\forall b\in Z(B)$. In other words, $\Tilde{\sigma}\in\Oz(B)$. By \cite[Theorem E]{CGWZ25} or \cite{LWZ25}, $\Tilde{\sigma}$ is of finite order. Since $\sigma=\Tilde{\sigma}|_A$, $\sigma$ is also of finite order. The conclusion follows from Proposition \ref{prop ufr and ore ext}.
\end{proof}

We would like to ask the following question.
\begin{ques}
    Are all noetherian PI AS-regular algebras noetherian UFRs?
\end{ques}

\section{Some noetherian AS-regular UFRs beyond PI cases}
In this section we give some non-PI AS-regular algebras which are also noetherian UFRs. As previously explained in the Introduction, it is possible to infer from \cite[Theorem 3.6]{LLR06} and Proposition \ref{prop ufr and ore ext} that all skew polynomial rings are noetherian UFRs. We now provide a direct proof of this assertion.

Nonzero normal elements in a prime ring are regular. For any regular normal element $z$ of a $\kk$-algebra $A$, $z$ induces a $\kk$-automorphism $\eta_z$ of $A$ given by $za=\eta_z(a)z$ for all $a \in A$. Let $S_{\mathcal{P}}$ be a skew polynomial ring of dimension $n$, then $S_{\mathcal{P}}$ is an $\mathbb{Z}^n$-graded $\kk$-algebra. Moreover, $x_i$ is a homogeneous regular normal element such that $\eta_{x_j}(x_i)=p_{ij}x_i$ for all $1\leqslant i,j\leqslant n$.

\begin{lem}\label{lem homo normal elements}
    Let $z$ be a nonzero normal elements of $S_{\mathcal{P}}$. If $z=z_1+z_2+\cdots +z_k$, where $z_j$ are nonzero homogeneous components of $z$ with respect to the $\mathbb{N}^n$-grading, then $z_j$ are normal elements such that $\eta_{z_j}=\eta_z$.
\end{lem}
\begin{proof}
    Note that each homogeneous element of $S_{\mathcal{P}}$ is of the form $k_{\alpha}x_1^{\alpha_1}x_2^{\alpha_2}\cdots x_n^{\alpha_n}$ for some $k_{\alpha}\in \kk$ and $\alpha=(\alpha_i,\cdots, \alpha_n)\in \mathbb{N}^n$. Since products of normal elements are still normal, all the homogeneous elements of $A$ are normal elements.
    
    Given that $zx_i=\eta_z(x_i)z$ for each $1\leqslant i\leqslant n$, we infer by comparing the homogeneous components that $z_jx_i=\eta_z(x_i)z_j$ holds for all $z_j$. Since each $z_j$ is a normal element, it follows that $z_jx_i=\eta_{z_j}(x_i)z_j$. Due to the fact that $S_{\mathcal{P}}$ is a domain, we deduce that $\eta_{z_j}(x_i)=\eta_z(x_i)$ for all $1\leqslant i\leqslant n$. Therefore, $\eta_{z_j}=\eta_z$ for every $j$.
\end{proof}

\begin{lem}\label{lem mutual eigenvalue}
    Let $z$ and $z'$ be two regular normal elements of a $\kk$-algebra $A$. If $\eta_z(z')=\lambda z'$ for some $\lambda\in \kk$, then $\eta_{z'}(z)=\lambda^{-1}z$.
\end{lem}
\begin{proof}
    By definition, $z'\eta_{z'}^{-1}(z)=zz'=\eta_z(z')z=\lambda z'z=\lambda \eta_{z'}(z)z'$. The result follows from that $z'$ is a regular element in $A$.
\end{proof}

\begin{lem}\label{lem normal action in skew poly}
    Let $S_{\mathcal{P}}=\kk_{\mathcal{P}}[x_1,\cdots x_n]$ be a skew polynomial ring, and let $z$ be a nonzero normal element of $S_{\mathcal{P}}$. Then there exists $\lambda_i\in\kk^{\times}$ such that $\eta_{x_i}(z)=\lambda_i z$ for any $1\leqslant i\leqslant n$.
\end{lem}
\begin{proof}
    Let $z=z_1+z_2+\cdots +z_k$, where $z_j$ are homogeneous components of $z$. Then $z_1=k_{\alpha}x_1^{\alpha_1}x_2^{\alpha_2}\cdots x_n^{\alpha_n}$ for some $k_{\alpha}\in \kk^{\times}$ and $\alpha_i\in \mathbb{N}$. Therefore, $\eta_z(x_i)=\eta_{z_1}(x_i)$, and
    $$\eta_{z_1}(x_i)z_1=z_1x_i= k_{\alpha}x_1^{\alpha_1}x_2^{\alpha_2}\cdots x_n^{\alpha_n} x_i=({p_{i1}}^{\alpha_1}\cdots {p_{in}}^{\alpha_n}) x_iz_1=\lambda_{i}x_iz_1$$  
    where $\lambda_{i}= {p_{i1}}^{\alpha_1}\cdots {p_{in}}^{\alpha_n} \in\kk^{\times}$. Hence $\eta_{z}(x_i)=\eta_{z_1}(x_i)=\lambda_{i}x_i$.
    As a result of Lemma \ref{lem mutual eigenvalue}, $\eta_{x_i}(z)=\lambda_{i}^{-1} z$.
\end{proof}
\begin{prop}\label{prop skew poly}
    All skew polynomial rings $S_{\mathcal{P}}$ are noetherian UFRs.
\end{prop}
\begin{proof}
    We prove by induction on the dimensions of skew polynomial rings.
    
    First note that $\kk[x]$ is the only skew polynomial ring of dimension $1$, which is known to be a noetherian UFD.
    
    Suppose that all the skew polynomial rings of dimension $n$ are noetherian UFRs. Let $S=\kk_{\mathcal{P}}[x_1,\cdots, x_{n+1}]$ be a skew polynomial ring. Then $S$ can be realized as an Ore extension of an $n$-dimensional skew polynomial ring. Precisely, let $R=\kk_{\mathcal{Q}}[x_1,\cdots, x_n]$, where $\mathcal{Q}=(p_{ij})_{1\leqslant i,j\leqslant n}\in \M_n(\kk^{\times})$. Then $S=R[x_{n+1}; \sigma]$, where $\sigma=\eta_{x_{n+1}}|_R$. Precisely, $\sigma(x_i)=p_{i,n+1}x_i$ for any $1\leqslant i\leqslant n$. As a skew polynomial ring of dimension $n$, $R$ is a UFR by induction hypothesis. According to Proposition \ref{prop ufr and ore ext}, it suffices to prove that every nonzero $\sigma$-prime ideal of $R$ contains a nonzero principal $\sigma$-ideal.

    Let $I$ be a nonzero $\sigma$-prime ideal of $R$. By \cite[5*]{GM75} there exists a prime ideal $J$ of $R$ such that $I=J\cap\sigma(J)\cap\cdots\cap \sigma^k(J)$ for some $k\in \mathbb{N}$. By induction hypothesis, there exists a nonzero normal element $z\in J$ such that $Rz$ is a prime ideal. It follows from Lemma \ref{lem normal action in skew poly} that there exists $\lambda\in\kk$ such that $\sigma(z)=\lambda z$. Therefore, $\sigma(Rz)=R\sigma (z)=Rz$, which implies that $Rz$ is a $\sigma$-ideal. Consequently $Rz\subseteq J\cap\sigma(J)\cap\cdots\cap \sigma^k(J)=I$, which completes the proof.
\end{proof}

As showed in Theorem \ref{thm 3-dim UFR}, noetherian PI AS-regular algebras of dimension $2$ are noetherian UFRs for trivial reason. As for those non-PI ones, we conclude that they are also noetherian UFRs following the classification of AS-regular algebras.
\begin{prop}
    All noetherian AS-regular algebras of dimension $2$ are noetherian UFRs.
\end{prop}

\begin{proof}
    Let $A$ be a noetherian AS-regular algebra of dimension $2$. Then either $A$ is a skew polynomial ring or $A=\kk\langle x,y\rangle/(xy-yx+x^k)$, where $k$ is a positive integer. It suffices to prove that algebras of the latter form are noetherian UFRs.
    
    Note that $A=\kk[x][y;\delta]$, where $\delta(x)=x^k$. By Proposition \ref{prop ufr and ore ext}, it suffices to prove that every nonzero $\delta$-prime ideal of $\kk[x]$ contains a nonzero principal $\delta$-ideal. By a direct calculation one sees easily that the only nonzero $\delta$-prime ideal of $\kk[x]$ is $(x)$, which is principal as desired.
\end{proof}

While the structures of higher-dimensional AS-regular algebras remain enigmatic, there is considerable knowledge regarding $3$-dimensional AS-regular algebras via classification. The classification reveals that $3$-dimensional noetherian AS-regular algebras are all domains. We propose the following question.

\begin{ques}
    Are all noetherian AS-regular algebras of dimension $3$ noetherian UFRs?
\end{ques}

\section*{Acknowledgments}

The authors are grateful to Ruipeng Zhu for helpful comments on an early draft of this paper. 

\bibliographystyle{amsalpha}
\bibliography{refs.bib}

\end{document}